\documentclass{amsart}

\usepackage{amssymb,latexsym,amsthm, amsmath, verbatim}
\usepackage{graphicx}

\ifx\pdfoutput\undefined
\usepackage{hyperref}
\else
\usepackage[pdftex,colorlinks=true,linkcolor=blue,urlcolor=blue]{hyperref}
\fi

\theoremstyle{plain}
\newtheorem{theorem}{Theorem}
\newtheorem{corollary}{Corollary}
\newtheorem{proposition}{Proposition}

\newtheorem*{thm}{Theorem}
\newtheorem{lemma}{Lemma}
\theoremstyle{definition}
\newtheorem{definition}{Definition}
\newtheorem{remark}{Remark}
\newtheorem*{question}{Question}

\newtheorem*{ack}{Acknowledgements}

\begin{document}
\title[Flat surfaces of finite area and ergodicity]
      {On the Ergodicity of Flat Surfaces of Finite Area}
\author{Rodrigo Trevi\~no}
\address{School of Mathematical Sciences\\
         Tel Aviv University\\
         Tel Aviv, Israel}
\email{rtrevino@math.cornell.edu}
\thanks{This work was supported by the Brin and Flagship Fellowships at the University of Maryland.}
\date{\today}
\begin{abstract}
We prove some ergodic theorems for flat surfaces of finite area. The first result concerns such surfaces whose Teichm\"{u}ller orbits are recurrent to a compact set of $SL(2,\mathbb{R})/SL(S,\alpha)$, where $SL(S,\alpha)$ is the Veech group of the surface. In this setting, this means that the translation flow on a flat surface can be renormalized through its Veech group. This result applies in particular to flat surfaces of infinite genus and finite area. Our second result is an criterion for ergodicity based on the control of deforming metric of a flat surface. Applied to  translation flows on compact surfaces, it improves and generalizes a theorem of Cheung and Eskin \cite{cheung-eskin}.
\end{abstract}
\maketitle

A flat surface is a two-dimensional oriented manifold $S$ endowed with a flat metric everywhere except on a set of ``bad points'' $\Sigma$ or singularities, which are forced to exist by the topology of the surface if the surface is of genus greater than one. Flat surfaces are inextricably connected to quadratic differentials since the latter give a Riemann surface a flat metric and a pair of transverse, measured foliations, called the vertical and horizontal foliations. If the foliations are orientable, which is not always the case, by considering them as flows we suddenly have a dynamical system, called the translation flow. In other words, since the foliations are orientable, the holomorphic 1-form or an Abelian differential $\alpha$ defining the quadratic differential defines a dynamical system on a surface. Thus one can try to derive dynamical and ergodic properties of the flow by studying properties of the Abelian differential. Although we can get two different flows by considering the horizontal or vertical foliations, from now on we shall assume the flow corresponds to the horizontal foliation, which can be thought as being defined along a global direction $\theta_\alpha\in S^1$. This flow preserves an absolutely continuous measure $\omega_\alpha$, singular at $\Sigma$, which is also defined by the Abelian differential. For a very thorough background on flat surfaces, see \cite{MasTab, zorich:intro}.

In the case when the surface is compact, the point of view of looking at a quadratic differential in order to derive properties of the dynamical system which it defines is a rather favorable one, as the ``right'' space of \emph{all} quadratic differentials on a fixed Riemann surface of genus $g$ is a finite dimensional space. This ``right'' space is the moduli space of quadratic differentials, or moduli space for short. It is the ``right'' space because it has many convenient properties. For example, it is the space of classes of conformally-equivalent flat metrics on a Riemann surface, it is a topological space homeomorphic to an open ball of dimension $6g-6$ (where $g$ is the genus of the surface), and it is equipped with an absolutely continuous $SL(2,\mathbb{R})$-invariant probability measure  \cite{Masur:IET, veech-teich}.  Properties of the translation flow on a compact flat surface can be derived from an associated dynamical system on moduli space, namely, the action of the diagonal subgroup of $SL(2,\mathbb{R})$ on the moduli space, also known as the Teichm\"{u}ller flow. This action has a geometric representation by giving a one-parameter family of flat surfaces which are obtained by deforming the flat metric in a very spcific manner. This deformation is called the Teichm\"{u}ller deformation.


The question of ergodicity of the translation flow on a compact flat surface is addressed by studying the beautiful interplay between the dynamics on a compact flat surface and that of the Teichm\"{u}ller flow on the moduli space of quadratic differentials. The relationship between the dynamics of the translation flow on a flat Riemann surface and that of the dynamics of the Teichm\"{u}ller flow on the moduli space of quadratic differentials is given by a well-known result of Masur:
\begin{thm}[Masur's criterion \cite{Masur:IET, masur:dim}]
If the translation flow on a flat Riemann surface is minimal and not uniquely ergodic, then the Teichm\"{u}ller orbit (of the class of that flat metric on which our translation flow is defined) leaves every compact set of the moduli space.
\end{thm}
Moreover, it is known that for almost every $\theta\in S^1$, the translation flow generated by $\alpha_\theta = e^{i\theta}\alpha$ is uniquely ergodic \cite{KMS} and the set of non-ergodic directions has Hausdorff dimension at most $\frac{1}{2}$ \cite{masur:dim}. It also holds that for almost every flat surface of genus greater than one, the translation flow is weakly mixing in almost every direction \cite{AvilaForni:WM}.

Let us denote a flat surface by $(S,\alpha)$. There are very special flat surfaces whose $SL(2,\mathbb{R})$ orbit in moduli space is three-dimensional. These flat surfaces are called Veech surfaces and what makes them special is a large collection of ``symmetries'' which preserve the flat structure. These symmetries renormalize the translation flow via the action of the diagonal subgroup of $SL(2,\mathbb{R})$, i.e., the action of the Teichm\"{u}ller flow. For these special surfaces it suffices to study their $SL(2,\mathbb{R})$ orbit to derive dynamical properties of the translation flow on it. In particular, the phase space of the orbit is a three-dimensional manifold, regardless of the genus of the surface, which is in high contrast with the dimension of the phase space in the typical case, since it grows linearly with the genus of the surface. In these special cases, Masur's criterion can be expressed as follows: if the translation flow on a Veech surface is minimal but not uniquely ergodic, then its Teichm\"{u}ller orbit leaves every compact subset of $SL(2,\mathbb{R})/SL(S,\alpha)$, where $SL(S,\alpha)$ is the large collection of symmetries already mentioned (and defined in \S \ref{subsec:SL2R}), called the \emph{Veech group} of the surface $S$. Veech surfaces have the additional property of satisfying the \emph{Veech dichotomy}: the translation flow in any direction $\theta$ on a Veech surface is either completely periodic or uniquely ergodic. By completely periodic we mean that all orbits which do not emanate from singularities are closed.

Since the dynamics of finite-genus flat surfaces are by now very well understood, there has been a recent surge in the study of the dynamics of the translation flow on flat surfaces of infinite genus \cite{chamanara,pascal-gabriela,infinite-staircase,bowman, hooper:lattice,hooper:measures,corinna, RalstTroub}. In this case, much of the structure from the finite-genus theory is lost. In particular, there is no well-defined notion of moduli spaces which allow us to carry out an analogous study and thus most results so far about the ergodicity of the translation flow on a flat surface of infinite genus are done in a case-by-case scenario. A common approach for all of the examples known and studied is the genus-independent approach already used in the finite genus case, that is, by exploiting the properties given by the Veech group of the surface. 

There are two types of infinite-genus flat surfaces that can be considered: flat surfaces with finite area and those with infinite area. At the moment, it seems there are more results for the ergodicity of the translation flow in the case of infinite area flat surfaces of infinite genus. Most of these surfaces are $\mathbb{Z}^d$ branched coverings of surfaces of finite area and one can recover some information about the dynamics on the cover from the dynamics on the finite-genus surface being covered. There are some infinite genus flat surfaces of finite area in the literature with non-trivial Veech groups, but there has been no unifying approach in these cases to prove ergodicity of the translation flow, although the results of \cite{hooper:measures} are a step in this direction.

In this paper we give a general proof of the ergodicity of the translation flow for infinite genus flat surfaces of finite area with sufficiently large Veech group. In spirit, our theorem is very much like Masur's criterion. The main result is the following. 
\begin{theorem}
\label{thm:main}
Let $(S,\alpha)$ be a flat surface of finite area whose Teichm\"{u}ller orbit does not leave every compact set of $SL(2,\mathbb{R})/SL(S,\alpha)$. Then the translation flow is ergodic with respect to Lebesgue measure.
\end{theorem}
In style, however, our theorem is different from Masur's criterion since the methods are quite different. In particular, it is not clear from this approach that unique ergodicity can be proved in this generality.

Theorem \ref{thm:main} applies to all of the known flat surfaces of infinite genus and finite area with non-trivial Veech groups \cite{chamanara,bowman,hooper:measures}. For some of these examples Theorem \ref{thm:main} gives ergodicity for the translation flow on surfaces where other methods provide no information about the ergodic properties of the flow and thus proves to be useful as a general tool, readily applicable to any new examples of flat surfaces of finite area and infinite genus with a non-trivial Veech group. 


Our second result is an ergodicity criterion based on the integrability of geometric quantities of a flat surface being deformed by the Teichm\"{u}ller deformation. In contrast to Theorem \ref{thm:main}, this result does not rely on the existence of a Veech group nor on any space of surfaces to which the flat surface belongs. We consider this criterion to be a form of geometric Khinchin-type criterion which implies unique ergodicity.

By $dist_t(x,y)$ we will refer to the distance between $x,y\in S$ with respect to the flat metric on $(S,\alpha_t)$, the one-parameter family of flat surfaces obtained from $(S,\alpha)$ by Teichm\"{u}ller deformation.
\begin{theorem}
\label{thm:integrability2}
Let $(S,\alpha)$ be a flat surface of finite area. Suppose that for any $\eta>0$ there exist a function $t\mapsto \varepsilon(t)>0$, a one-parameter family of subsets
$$S_{\varepsilon(t),t} = \bigsqcup_{i=1}^{C_t}S_t^i$$ 
of $S$ made up of $C_t < \infty$ path-connected components, each homeomorphic to a closed orientable surface with boundary, and functions $t\mapsto \mathcal{D}_t^i>0$, for $1\leq i \leq C_t$, such that for 
$$\Gamma_t^{i,j} = \{\mbox{paths connecting }\partial S_t^i \mbox{ to }\partial S_t^j\}$$
and
\begin{equation}
\label{eqn:systole}
\delta_t = \min_{i\neq j} \sup_{\gamma\in\Gamma_t^{i,j} }\mbox{dist}_t(\gamma,\Sigma)
\end{equation}
the following hold:
\begin{enumerate}
\item  $\omega_\alpha(S\backslash S_{\varepsilon(t), t}) < \eta$ for all $t>0$,
\item $\mbox{dist}_t(\partial S_{\varepsilon(t),t},\Sigma) > \varepsilon(t)$ for all $t>0$,
\item the diameter of each $S_t^i$, measured with respect to the flat metric on $(S,\alpha_t)$, is bounded by $\mathcal{D}_t^i$ and
\begin{equation}
\label{eqn:integrability2}
\int_0^\infty \left( \varepsilon(t)^{-2}\sum_{i=1}^{C_t}\mathcal{D}_t^i + \frac{C_t-1}{\delta_t}\right)^{-2}\, dt = +\infty.
\end{equation}
\end{enumerate}
Moreover, suppose the set of points whose translation trajectories leave every compact subset of $S$ has zero measure. Then the translation flow is ergodic.
\end{theorem}

In the special case that $(S,\alpha)$ is a compact surface, the theorem above can be expressed as follows. We denote now by $\delta_t$ the \emph{systole} of a surface evolving under the Teichm\"{u}ller flow. The systole is the length (measured with the flat metric on $(S,\alpha_t)$, the orbit of $(S,\alpha)$ under the Teichm\"{u}ller flow) of the shortest homotopically non-trivial closed curve of $S$.
\begin{theorem}
\label{thm:integrability1}
Let $S$ be a compact flat surface of finite area. If
\begin{equation}
\label{eqn:integrability}
\int_0^\infty\delta_t^2\,dt = \infty,
\end{equation}
then the translation flow is uniquely ergodic.
\end{theorem}
It seems highly unlikely that (\ref{eqn:integrability}) is a necessary condition for unique ergodicity but we do not know of any counterexamples. The upgrade from ergodicity in Theorem \ref{thm:integrability2} to unique ergodicity in Theorem \ref{thm:integrability1} is obtained through an argument suggested by Jon Chaika, for which we are grateful.

Criteria like Theorem \ref{thm:integrability1} for compact flat surfaces are already found in the literature, notably the Veech-Boshernitzan criterion for unique ergodicity \cite{veech:boshernitzan} which is stated for interval exchange transformations. Our criterion here is more general than the Veech-Boshernitzan criterion and has the advantage of being more transparent since it involves more explicit geometric quantities of the surface evolving under the Teichm\"{u}ller flow. Since the diameter can roughly be expressed in terms of shortest homotopically nontrivial closed curve when the surface is compact, we can formulate the criteria in these terms.

Theorem \ref{thm:integrability1} is particularly interesting when $\lim \delta_t = 0$. Otherwise, $\limsup \delta_t>0$ is treated by Masur's Criterion. Indeed, Theorem \ref{eqn:integrability} yields a new proof for Masur's Criterion. Define $d'(t) = -\log\delta_t'$, where $\delta_t'$ is the length of the shortest saddle connection of $(S,\alpha)$ evolving under the Teichm\"{u}ller flow. Cheung and Eskin \cite{cheung-eskin} proved the following. 
\begin{thm}[\cite{cheung-eskin}]
\label{thm:CheungEskin}
There is an $\epsilon>0$ such that if $d'(t) < \epsilon \log t + C$ for some $C$ and for all $t>0$, then the translation flow is uniquely ergodic.
\end{thm}

It follows from Theorem \ref{thm:integrability1} that it suffices to pick $\epsilon = \frac{1}{2}$ in the above theorem. Indeed, we have that $d(t)\equiv -\log\delta_t \leq d'(t)$ for all $t$, and if 
$$d(t) < \frac{1}{2}\log t + C$$
for some $C>0$, then
$$C't^{-\frac{1}{2}}<\delta_t,$$
where $C'$ is some constant. Squaring and integrating both sides with respect to $t$, the hypothesis of Theorem \ref{thm:integrability1} is satisfied.

\begin{corollary}
Suppose $d'(t) < \frac{1}{2} \log t + C$ for some $C$ and for all $t>0$. Then the translation flow is uniquely ergodic.
\end{corollary}
Let us recall the \emph{Logarithmic Law for Geodesics} for flat surfaces: let $\mbox{dist}(\alpha,g_t\alpha)$ be the distance in a stratum of the moduli space between a surface carrying an Abelian differential $\alpha$ and the one carrying $g_t\alpha$, which is its orbit under the Teichm\"{u}ller flow. The quantities $\mbox{dist}(\alpha,g_t\alpha)$ and $d'(t)$ are roughly the same (see \cite{masur:loglaw} for the precise relationship).
\begin{thm}[Logarithmic Law for Geodesics in Moduli Space \cite{masur:loglaw}]
For any Abelian differential $\alpha$ on a compact flat surface of finite area, for almost every $\theta\in S^1$,
$$\limsup_{t>0}\frac{\mbox{dist}(\alpha_\theta,g_t\alpha_\theta)}{\log t} = \frac{1}{2}.$$
\end{thm}
The theorem of Cheung and Eskin, with $\epsilon = \frac{1}{2}$, implies that there are flat surfaces which satisfy Masur's Logarithmic Law and whose translation flow is uniquely ergodic. Such surfaces indeed form a full measure set since the set of surfaces carrying uniquely ergodic foliations form a full measure set, as do the ones satisfying Masur's Logarithmic Law. The significance of Theorem \ref{thm:integrability1} is that it draws a sharper line between compact flat surfaces carrying uniquely ergodic foliations and those that do not, especially those which satisfy Masur's Logarithmic Law. We do not know how sharp the value $\epsilon = \frac{1}{2}$ is nor the connection this value may have with Masur's Logarithmic Law.

\begin{question}
Is there a compact flat surface which satisfies Masur's Logarithmic Law but whose horizontal foliation is not uniquely ergodic?
\end{question}

We hope that the results here will be more motivation to understand flat surfaces of infinite genus and finite area. It is becoming increasingly clear that understanding the geometry of non-compact flat surfaces of finite area is crucial to knowing ergodic properties of translation flows defined on them. The inconvenient fact that there is no known moduli space for such surfaces forces us to look for a good class of surfaces whose geometry can be more deeply understood. Of particular interest are parabolic surfaces (that is, surfaces on which there is no Green's function) since they share many of the analytic properties of compact surfaces. Such analytic properties may become useful in analytic approaches such as the ones in this paper. Moreover, parabolic surfaces guarantee that the dynamics of the translation flow are not as trivial as possible (see Remark \ref{rem:measure}). 

The approach in this paper is inspired by proof of the spectral gap for the Kontsevich-Zorich cocycle; compare \cite[Lemma 2.1']{forni:deviation} with Lemma \ref{lem:pos} here, which is a crucial ingredient for all ergodicity theorems. We hope that the ideas here may be useful in proving a criterion for a spectral gap for flat surfaces of infinite genus and finite area using Forni's method \cite[Lemma 2.2]{forni:deviation}. We hope to visit this question in future work.

It is worth pointing out that the techniques in this paper used to establish the ergodicity criteria do not readily yield criteria for weak mixing. We hope to extend the techniques of the proofs in future work to yield criteria for weak mixing.

Section 1 gives background on flat surfaces from a geometric and analytic point of view, as well as background on Veech groups. Section 2 deals with proving the main result, Theorem \ref{thm:main}. In Section 3 we prove the integrability criterion for compact flat surfaces, Theorem \ref{thm:integrability1}.

\begin{ack}
This work was done as part of my PhD dissertation at the University of Maryland, College Park. I am infinitely grateful to my advisor, Giovanni Forni, for suggesting this problem, for his endless support, motivation and patience, and for many very enlightening discussions. This would have never been possible without him.

I want to thank Barak Weiss, Pat Hooper and Josh Bowman for many useful comments and conversations. Finally I want to thank the hospitality of the Universitat Polit\`{e}cnica de Catalunya where some of this work was carried out.
\end{ack}

\section{Flat Surfaces and Veech Groups}

\subsection{Flat structures}

Let $S$ be a Riemann surface with no boundary and $\Sigma\subset \bar{S}$ a discrete set of points, where the surface $\bar{S}$ is some compactification of $S$. The manifold $S$ is a flat surface if it carries an atlas $\{(\mathcal{U}_i,\varphi_i)\}_i$ with $\mathcal{U}_\alpha,\, \mathcal{U}_\beta\subset S\backslash \Sigma$ such that for any two charts $(\mathcal{U}_\alpha,\varphi_\alpha)$ and $(\mathcal{U}_\beta,\varphi_\beta)$, $\varphi_\alpha\circ\varphi^{-1}_\beta(z) = \pm z + c_{\alpha\beta}$ for $z\in\mathcal{U}_\alpha\cap\mathcal{U}_\beta$. If $\varphi_\alpha\circ\varphi^{-1}_\beta(z) = z + c_{\alpha\beta}$ for all $z$ and $\alpha,\,\beta$, then $S$ is a \emph{translation surface}. Otherwise it is a \emph{half-translation surface}. Here we will only be interested in translation surfaces since half-translation surfaces can be studied by passing to an appropriate double cover where they become translation surfaces.

The points which make up $\Sigma$ are the singularities of $S$. Any compact  translation surface $S$ of genus greater that one must, by the Gauss-Bonnet theorem, have overall negative curvature. Since a translation surface has a flat metric everywhere on $S\backslash \Sigma$, any surface of genus greater than one must have its negative curvature concentrated on $\Sigma$. Thus at a point $p\in\Sigma$ the metric can be written in polar coordinates $(r,\theta)$ as $\sqrt{dr^2+(ar\,d\theta)^2)}$, where $2\pi a$ is the cone angle at $p$.

The complex structure of a translation surface can also be completely obtained by an \emph{Abelian differential}, i.e., a holomorphic 1-form. In local coordinates away from $\Sigma$ any Abelian differential can be written as $\alpha = \phi(z)\, dz$, with $\phi$ a holomorphic function, and the metric can be written locally as $R_\alpha = |\alpha| |dz|$ while the area form is given on $S\backslash (\Sigma\cap S)$ by $\omega_\alpha = \Re\, (\alpha)\wedge\Im\,(\alpha)$. Any Abelian differential $\alpha$ comes with a pair of transverse measured foliations, the horizontal and vertical foliations, $\mathcal{F}^h_\alpha$ and $\mathcal{F}^v_\alpha$. They are the foliations generated by the distributions $\mathrm{Ker}\, \Im(\alpha)$ and $ \mathrm{Ker}\,\Re(\alpha)$, respectively.

A flat surface will be denoted as $(S,\alpha)$ which emphasizes the metric and foliated structure imposed on the topological surface $S$ by the Abelian differential $\alpha$. The flat surface $(S,\alpha_\theta)$, where $\alpha_\theta = e^{i\theta}\alpha$, carries the same metric as the flat surface $(S,\alpha)$, but their foliations differ. The foliations on $(S,\alpha_\theta)$ are simply obtained by ``rotating'' the foliations on $(S,\alpha)$ by the angle $\theta$. Sometimes we may refer to $S$ as a flat surface without specification of any Abelian differential. In such case, we mean that we are considering $(S,\alpha_\theta)$ for some $\alpha$ and all $\theta\in S^1$.

For the flat surface $(S,\alpha)$ we will always assume that $\bar{S}$ is the compactification of $S$ obtained through the metric completion of $S$ with respect to the flat metric compatible with $\alpha$.

A \emph{regular} leaf for the vertical or horizontal foliation is a leaf which does not limit to a point in $\Sigma$, i.e., a singularity of $\alpha$. Otherwise it is called \emph{singular}. A \emph{generalized saddle connection} of $\alpha$ is a singular leaf of the vertical or horizontal foliation which limits to two singularities. We remark that in the case of non-compact surfaces, the set of singularities also includes the ideal boundary of $S$, a feature that is not present for compact surfaces. A \emph{regular saddle connection} is a leaf which joins two singularities in $S$. A \emph{cross cut} has the property that that it leaves every compact subset of $S$ in both directions. In other words, a cross cut is a trajectory from the ideal boundary to itself. As such, in some cases, cross cuts may be arbitrarily long, even of infinite length. 

\begin{remark}
\label{rem:measure}
Since we are interested in studying the properties of flows on flat surfaces of finite area, \textbf{we assume throughout the paper that the set of cross cuts forms a set of measure zero} for all surfaces studied, as otherwise this would create an obstruction to ergodicity. This is only relevant for non-compact surfaces. If we assume that our surface is parabolic (i.e. $S$ has no Green's function or, equivalently, the harmonic measure of the ideal boundary of $S$ vanishes) then \cite[Theorems 13.1 and 24.4]{strebel:book} guarantee that almost every point has a recurrent orbit. Parabolicity in a sense implies that the ideal boundary of a non-compact surface is ``small enough'' to have most orbits be recurrent. It can be shown that the non-compact surfaces from \cite{chamanara,bowman,hooper:measures}, for example, are parabolic through the extremal distance criterion \cite[IV.15A]{AS}.
\end{remark}

We denote the set of generalized saddle connections of an Abelian differential on a surface $S$ by $\mathrm{SC}(S,\alpha)$. The horizontal or vertical foliations of an abelian differential are \emph{periodic} if all but perhaps the singular leaves are closed. In such case, by considering $S\backslash \mathrm{SC}(S,\alpha)$, the surface decomposes into a union of cylinders bounded by saddle connections and each cylinder is foliated by homotopically-equivalent closed leaves of the foliation. It may be possible for a surface $S\backslash \mathrm{SC}(S,\alpha)$ to decompose as the disjoint union of periodic components (cylinders), and minimal components.

The length of a saddle connection or homotopically non-trivial simple closed curve with respect to the flat metric can be computed through its horizontal and vertical components. Specifically, for $\ell \in SC(S,\alpha)$, or for $\ell$ a homotopically non-trivial simple closed curve,
\begin{equation}
\label{eqn:length}
\mathrm{length}_\alpha(\ell)^2 =  H_\alpha(\ell)^2 + V_\alpha(\ell)^2,
\end{equation}
where $H_\alpha(\ell) = \int_\ell |\Re(\alpha)|$ and $V_\alpha(\ell) = \int_\ell |\Im(\alpha)|$.


In this paper we deal with flat surfaces of infinite genus and finite area. For such surfaces the set of singularities $\Sigma$ not only consists of finite-angle singularities as in the compact case, but in addition of singularities of infinite angle. This will be of no consequence in the present analysis. These surfaces also carry a translation structure just as in the finite genus case and therefore a (singular) flat metric given by an Abelian differential $\alpha$. The requirement that the surface have finite area is equivalent to the requirement that the norm of the Abelian differential, $\|\alpha\| = \int_S |\alpha|^2\, \omega_\alpha$, is finite.

The vector fields $X$ and $Y$ of norm 1 which, respectively, are tangent to the foliations $\mathcal{F}^{h,v}_\alpha$, commute and in addition have the following properties \cite{forni:1997}:
\begin{enumerate}
\item $\{X,Y\}$ is an orthonormal frame for the tangent bundle $TS$ on $S\backslash\Sigma$ with respect to the metric $R_\alpha$.
\item $X$ and $Y$ preserve the smooth area form $\omega_\alpha$, thus $\eta_X \equiv \imath_X\omega_\alpha$ and $\eta_Y \equiv -\imath_Y\omega_\alpha$ are closed, smooth 1-forms on $S\backslash \Sigma$.
\item $\eta_X$ and $\eta_Y$ generate the measured foliations $\mathcal{F}^{h,v}_\alpha$ on $S\backslash \Sigma$.
\end{enumerate}
The complex structure provided by the Abelian differential $\alpha$ also defines spaces of functions compatible with the induced foliations and the vector fields $X$ and $Y$. We define
\begin{equation}
\label{eqn:L2}
L^2_\alpha(S) = \left\{ u : \int_S |u|^2\,\omega_\alpha\equiv \|u\|^2 <\infty \right\}
\end{equation}
to be the weighted $L^2$ spaces of $S$. These spaces have a natural structure of Hilbert spaces with inner product $(\cdot,\cdot)_\alpha$ defined as
$$(u,v)_\alpha \equiv \int_S u\bar{v}\,\omega_\alpha$$
which satisfies, by the invariance of the of $\omega_\alpha$ under $X$ and $Y$,
\begin{equation}
\label{eqn:comm}
(Xu,v)_\alpha = -(u,Xv)_\alpha\mbox{ and }(Yu,v)_\alpha = -(u,Yv)_\alpha.
\end{equation}
Define the $s-$norm to be
\begin{equation}
\label{eqn:sobolev}
\|u\|_s^2 \equiv  \sum_{i+j\leq s}\|X^iY^ju\|^2.
\end{equation}
Let $H^s_\alpha(S)$ be the completion of the set of smooth functions with finite $\|\cdot\|_s$ norm. We denote by $H^{-s}_\alpha(S)$ the dual space of $H^s_\alpha(S)$.
From the vector fields $X$ and $Y$, we construct the $\emph{Cauchy-Riemann operators}$
\begin{equation}
\label{eqn:CRops}
\partial^\pm_\alpha\equiv X\pm i Y,
\end{equation}
the kernels of which contain the meromorphic, respectively anti-meromorphic, functions which are elements of $L^2_\alpha(S)$. As shown in \cite[Proposition 3.2]{forni:1997}, it follows from (\ref{eqn:comm}) that $(\partial^\pm_\alpha)^*$ are extensions of $-\partial_\alpha^\mp$. It follows by Hilbert space theory that we have the orthogonal splitting
\begin{equation}
\label{eqn:L2split}
L^2_\alpha(S) = \mathrm{Range}(\partial^\pm_\alpha)\oplus^\perp \mbox{ Ker}\,(\partial^\mp_\alpha).
\end{equation}
Finally, the Dirichlet form $Q_\alpha:H^1_\alpha(S)\times H^1_\alpha(S)\rightarrow\mathbb{C}$ is defined as
$$Q_\alpha(u,v) = (Xu,Xv)_\alpha + (Yu,Yv)_\alpha = (\partial^\pm_\alpha u,\partial^\pm_\alpha v)_\alpha.$$
The \emph{Dirichlet norm} of a function $u$ is defined to be $Q_\alpha(u) \equiv Q_\alpha(u,u)$.

\subsection{$SL(2,\mathbb{R})$ action}
\label{subsec:SL2R}

Let $(S,\alpha)$ denote a surface $S$ with a complex structure given by an Abelian differential $\alpha$. There is a well-defined action of the group $SL(2,\mathbb{R})$ on $(S,\alpha)$. For $A\in SL(2,\mathbb{R})$, we define $A\cdot(S,\alpha)$ to be the surface $(S,\alpha)$ with charts post-composed with the action of $A$ on $\mathbb{R}^2$.

The stabilizer of this action is denoted by $Stab(S,\alpha)$ and its image in $PSL(2,\mathbb{R})$ is called the \emph{Veech group} of $(S,\alpha)$. It is usually denoted by $SL(S,\alpha)$ or \emph{Aff}$(S,\alpha)$ since it coincides with the group of derivatives of affine diffeomorphisms (with respect to $\alpha$) of $S$. In other words, if $r\in SL(S,\alpha)$, then there exists a unique affine diffeomorphism $f_r$ with constant derivative $Df_r$ such that the action of $Df_r$ on the complex structure of $(S,\alpha)$ coincides with that of $r$. Such diffeomorphisms will be called \emph{Teichm\"{u}ller maps}.

When $S$ is compact, the Veech group $SL(S,\alpha)$ is always a discrete subgroup and, when $SL(S,\alpha)$ is a lattice, $(S,\alpha)$ is called a \emph{Veech surface}. Usually one expects the Veech group of a surface to be trivial. Thus, surfaces with non-trivial Veech groups turn out to be quite interesting (and are hard to find). The $SL(2,\mathbb{R})$-orbit of $(S,\alpha)$, denoted by $D_{(S,\alpha)}$, is isometric to the unit tangent bundle of the Poincar\'{e} disk $\mathbb{H}$, and is called the \emph{Teichm\"{u}ller disk} of $(S,\alpha)$. The Veech group $SL(S,\alpha)$ acts on $D_{(S,\alpha)}$ by isometries of the hyperbolic metric. The quotient of the Teichm\"{u}ller disk of $(S,\alpha)$ by its Veech group is denoted by
$$H_{(S,\alpha)}\equiv\mathbb{H}/SL(S,\alpha),$$
where $\mathbb{H} = SL(2,\mathbb{R})/SO(2,\mathbb{R})$. The projection map will be denoted by 
$$\Pi_{(S,\alpha)}:D_{(S,\alpha)}\rightarrow H_{(S,\alpha)}.$$ 
We will identify deformations of $(S,\alpha)$ by elements of $SL(w,\mathbb{R})$ with the elements themselves; this makes it natural to talk about the $SL(2,\mathbb{R})$ or Teichm\"{u}ller orbits of $(S,\alpha)$ in $SL(2,\mathbb{R})/SL(S,\alpha)$.

The disk $H_{(S,\alpha)}$ has finite area if, and only if, $(S,\alpha)$ is a Veech surface. However, if $(S,\alpha)$ is compact, $H_{(S,\alpha)}$ is never compact. It is not known whether there exists a flat surface of finite area and infinite genus whose Veech group is not discrete.

It is natural to talk about \emph{elliptic, parabolic,} and \emph{hyperbolic} elements of $SL(2,\mathbb{R})$ corresponding, respectively, to elements with zero, one, and two distinct real eigenvalues. Elliptic elements are conjugate to the elements of the subgroup $SO(2,\mathbb{R})$. Elements of $SO(2,\mathbb{R})$, rotations, will be denoted by $r_\theta$, where $\theta$ is the angle of rotation. Parabolic elements are, in a conveniently-rotated coordinate system, of the form
$$h^t = \left( \begin{array}{cc}
1 & t \\
0 & 1 \end{array} \right) \hspace{.35in} \mbox{and}\hspace{.35in} h_s = \left( \begin{array}{cc}
1 & 0 \\
s & 1 \end{array} \right)$$
for $s,t\in\mathbb{R}$. Parabolic elements generate both parabolic elements and hyperbolic elements. Associated to every parabolic element there corresponds a unique invariant direction corresponding to its eigenvector and we say that the parabolic element \emph{fixes} this direction. Any direction invariant by a hyperbolic element is also said to be fixed by it.

The diagonal subgroup
$$g_t \equiv \left\langle \left( \begin{array}{cc}
e^{-t} & 0 \\
0 & e^{t} \end{array}\right): t\in\mathbb{R} \right\rangle$$
is an important subgroup of $SL(2,\mathbb{R})$. Its action on the Teichm\"{u}ller disk of a flat surface is called the \emph{Teichm\"{u}ller geodesic flow} since it minimizes distances between two points in the Teichm\"{u}ller disk of a flat surface. Its action on the complex structure of $(S,\alpha)$ is also referred to as Teichm\"{u}ller deformation.


\section{Ergodicity for recurrent Teichm\"{u}ller orbits}

Recall that the complex structure of any translation surface is given by an Abelian differential $\alpha$ which defines a commuting pair of vector fields $X$ and $Y$ of norm 1. Let
\begin{equation}
\label{eqn:CStopps}
\partial_t^\pm \equiv e^{t}X \pm ie^{-t}Y = X_t \pm i Y_t
\end{equation}
be the one-parameter family of Cauchy-Riemann operators defined for the complex structure given by the Abelian differential
$$\alpha_t = e^{-t}\Re(\alpha) + i e^{t}\Im(\alpha).$$
In other words, one can think of the operators $\partial^\pm_t$ as the Cauchy-Riemann operators of the surface $g_t\cdot (S,\alpha) = (S,\alpha_t)$. To be consistent with the notation of (\ref{eqn:CRops}), we make the identification $\partial^\pm_t\equiv \partial^\pm_{\alpha_t}$.

Note that the volume form $\omega_{\alpha_t}$ given by $\alpha_t$ is invariant, i.e., $\omega_{\alpha_t} = \omega_\alpha$ for all $t$, and thus the Hilbert space of square integrable functions with respect to $\omega_{\alpha_t}$ is the same for all $t$ (and so are all derived Sobolev spaces $H_{\alpha_t}^s(S)$). There is, however, a one-parameter splitting of $L^2_\alpha$ as in (\ref{eqn:L2split}):
\begin{equation}
\label{eqn:split}
L^2_\alpha(S) = \mathrm{Range}(\partial^\pm_t)\oplus^\perp_t \mbox{ Ker}\,(\partial^\mp_t).
\end{equation}
Thus, for any function $u\in L^2(S)$ and for any $t\in \mathbb{R}$ there exist functions $v_t^\pm\in H^1_\alpha(S)$, meromorphic functions $m_t^-$ and anti-meromorphic functions $m_t^+$ such that
\begin{equation}
\label{eqn:L2tsplit}
u = \partial^+_tv^+_t + m_t^- = \partial ^-_t v^-_t + m^+_t.
\end{equation}
If the surface is compact, the spaces $\mbox{ Ker}\,(\partial^\pm_t)$ are finite dimensional by the Riemann-Roch theorem. For surfaces of infinite genus this is not the case necessarily, but this fact is irrelevant in the discussion.

If we chose each $v_t^\pm$ to have zero average, then the one parameter families $v_t^\pm$ are smooth. So we will assume this without loss of generality. Finally, it is easy to verify that
\begin{equation}
\label{eqn:id}
\overline{\partial^\pm_t f} = \partial^\mp_t \bar{f}.
\end{equation}

To address the issue of ergodicity of the flows (or foliations) generated by $X$ and $Y$, we are interested in studying functions which are $X$-invariant (or $Y$-invariant). Note that if $u\in L^2(S)$ is an $X$-invariant function, i.e., $Xu = 0$, by considering its real part we can study $X$-invariant functions while assuming they are real valued. We also assume that if $u\in L^2$ is a real-valued, invariant function, then $u\in L^\infty$. Indeed, for any invariant $u$, the set $A_u(r) \equiv \{x\in S: |u(x)| > r\}$ is invariant for any $r$, so for our purposes we can work with the function $u' = \chi_{S\backslash A_u(r)}u + \chi_{A_u(r)}$ for some $r$, which implies $\| u'\|_{\infty}<\infty$.

Let us now consider an arbitrary real function $u\in L^2_\alpha(S)$. Since $u$ is real-valued, using (\ref{eqn:L2tsplit}) and (\ref{eqn:id}),
$$u = \partial^+_tv^+_t + m^-_t = \overline{\partial^-_t v^-_t} +\overline{m^+_t} = \partial^+_t\bar{v}^-_t + \overline{m^+_t}= \bar{u},$$
from which, by (\ref{eqn:split}), it follows that 
\begin{equation}
\label{eqn:meroPart}
m^+_t = \overline{m^-_t}.
\end{equation}
Moreover, it also follows from (\ref{eqn:L2tsplit}) and (\ref{eqn:id}) that
$$\partial^+_t(v^+_t-\bar{v}_t^-) = 0$$
in $L^2_\alpha(S)$. In other words, the function $v^+_t-\bar{v}_t^-\in H^1_\alpha(S)$ has zero Dirichlet norm, so it is in the kernel of $Q_\alpha$. Since the kernel of $Q_\alpha$ consists of constant functions and $v_t^\pm$ were chosen to be of zero average, we have that
\begin{equation}
\label{eqn:v's}
v^+_t = \overline{v^-_t}.
\end{equation}

\begin{lemma}
\label{lem:ids}
Let $u\in L^2_\alpha$ be a real-valued, $X$-invariant function on a flat surface of finite area. Then, writing $u$ as in (\ref{eqn:L2tsplit}), we have that $v^+_t$ (and thus $v^-_t$) is purely imaginary and that
$$v^+_t = -v^-_t.$$
\end{lemma}

\begin{proof}
By applying $\partial^\pm_t$ to the decomposition (\ref{eqn:L2tsplit}) we obtain that
\begin{equation}
\label{eqn:1}
\begin{split}
\triangle_t v^+_t & = (\partial^-_t)^2v^-_t + \partial^-_tm_t^+ \\
\triangle_t v^-_t & = (\partial^+_t)^2v^+_t + \partial^+_tm_t^-
\end{split}
\end{equation}
in $H^{-1}_\alpha(S),$ where $\triangle_t = \partial^\pm_t\partial^\mp_t$ is the Laplacian with respect to the complex structure given by $\alpha_t$. Since $X_t = \frac{1}{2}(\partial_t^+ + \partial_t^-)$, then $Xu=0$ implies, by (\ref{eqn:L2tsplit}),
\begin{equation}
\label{eqn:2}
(\partial_t^+)^2v^+_t + \partial^+_tm^-_t + \triangle_tv^+_t = 0 \hspace{.25in}\mbox{ and }\hspace{.25in} (\partial_t^-)^2v^-_t + \partial^-_tm^+_t + \triangle_tv^-_t = 0.
\end{equation}
Putting (\ref{eqn:1}) and (\ref{eqn:2}) together,
$$\triangle_tv_t^+-\partial^-_tm^+_t = -\triangle_tv_t^- - \partial^-_tm^+_t,$$
which implies $\triangle_t(v_t^++v^-_t) = \triangle_t(2\Re(v^+_t)) = 0$. In other words, $\Re(v^+_t)\in H^1_\alpha(S)$ is a harmonic function. Moreover, since $v^\pm_t\in H^1_\alpha(S)$, 
$$Q_\alpha(\Re(v^+_t)) = (\partial^\pm_t \Re(v^+_t),\partial^\pm_t \Re(v^+_t)) = 0,$$ 
i.e., the Dirichlet norm of $\Re(v^+_t)$ is zero. Since the kernel of $Q_\alpha$ consists of constant functions and $v^\pm_t$ can be chosen to be of zero average (without loss of generality), $\Re(v^+_t) = 0$ and the result follows from (\ref{eqn:v's}).
\end{proof}

Using (\ref{eqn:meroPart}) and Lemma \ref{lem:ids}, we can compute the evolution of the norm of $m^\pm_t$.

\begin{lemma}
\label{lem:pos}
Under the splitting (\ref{eqn:L2tsplit}) for a real-valued, $X$-invariant function $u$ on a flat surface of finite area, the evolution of the norm of $m^\pm_t$ is described by
$$ \frac{d}{dt} \|m^\pm_t\|^2 = 4\|\Im(m^\pm_t)\|^2. $$
\end{lemma}
\begin{proof}
From (\ref{eqn:CStopps}) it follows by a straight-forward calculation that $\frac{d}{dt}\partial^\pm_t = \partial^\mp_t$. We perform the calculation $m^+_t$; the case for $m^-_t$ is essentially the same.
\begin{eqnarray*}
\frac{1}{2}\frac{d}{dt}\|m^+_t\|^2 &=& \mathrm{Re}\, \left(\frac{d}{dt}m^+_t,m^+_t\right) = \mathrm{Re}\,\left(\frac{d}{dt}(u-\partial^-_tv^-_t), m^+_t\right) \\
&=& -\mathrm{Re}\, \left(\frac{d}{dt}\partial^-_tv^-_t, m^+_t\right) = -\mathrm{Re}\,\left(\partial^+_t v^-_t + \partial^-_t\dot{v}^-_t, m^+_t\right) \\
&=& \mathrm{Re}\,(\partial^+_tv^+_t,m^+_t) = \mathrm{Re}\,(u-m^-_t,m^+_t) \\
&=& \|m^+_t\|^2 - \mathrm{Re}\, \int_S (m^+_t)^2\,\omega_\alpha  = 2\|\Im(m^+_t)\|^2.
\end{eqnarray*}
\end{proof}

\begin{definition}
\label{def:recurrent}
The $g_t$ orbit of $(S,\alpha)$ is \emph{recurrent} if for any $\varepsilon>0$ there is an $s_\varepsilon\in\mathbb{R}^+$ and an element $r\in SL(S,\alpha)$, $r\neq \mbox{Id}$, such that the distance between $g_{s_\varepsilon}\cdot(S,\alpha)$ and $r\cdot(S,\alpha)$ is less than $\varepsilon$ in $D_{(S,\alpha)}$.
\end{definition}
This definition gives the usual definition, from the point of view of topological dynamics, of a recurrent orbit on $H_{(S,\alpha)}$. We use this definition since it will be more useful in the proof of ergodicity.

\begin{remark}
\label{rem:recurrent2}
If $(S,\alpha)$ is $g_t$-recurrent, then for any sequence of $\varepsilon_i\rightarrow 0$ there is a sequence of angles $\theta_i$ and times $t_i\rightarrow\infty$ such that the distance between $g_{t_i}\cdot(S,\alpha)$ and $g_{t_i}r_{\theta_i}(S,\alpha)$ is less than $\varepsilon_i$ and $g_{t_i}r_{\theta_i}\in SL(S,\alpha)$. As such, it follows that $r_{\theta_i}\rightarrow \mathrm{Id}$, i.e., $\theta_i\rightarrow 0$. Indeed, since the $SL(2,\mathbb{R})$ orbit of $(S,\alpha)$ is isometric to the unit tangent bundle of the Poincar\'e disk, i.e., a simply connected surface with constant sectional curvature $\kappa = -4$, it follows from the hyperbolic law of sines that 
$$|\sin(\theta_i)| \leq \frac{\sinh(2\varepsilon_i)}{\sinh(2t_i)}.$$
Moreover, for each $k$ there exists an $e_k>0$ such that for all $s\in(t_k-e_k,t_k+e_k)$ we have that 
$$\mbox{dist}(g_s,g_{t_k}r_{\theta_k}) < \varepsilon_k,$$
all of which can be taken to be the $s_{\varepsilon_k}$ from Definition \ref{def:recurrent}. Finally, since $g_{t_i}r_{\theta_i}\in SL(S,\alpha)$, there exists a sequence of affine diffeomorphisms $f_i$ such that $g_{t_i}r_{\theta_i}=Df_i\in SL(S,\alpha)$.
\end{remark}
For a flat surface $(S,\alpha)$ with a recurrent $g_t$ orbit, we will call a sequence of quadruples 
$$\{(t_i, \theta_i, \varepsilon_i, f_i)\}_{i=1}^\infty \in (\mathbb{R}^+\times S^1\times \mathbb{R}^+\times SL(S,\alpha))^\mathbb{N}$$
as in the above remark the \emph{recurrent data} of $(S,\alpha)$. We can assume without loss of generality that $\varepsilon_{i+1}< \varepsilon_i$ for all $i$.

\begin{lemma}
\label{lem:cylinder}
Let $(S,\alpha)$ be a flat surface of finite area whose $g_t$-orbit is recurrent. Then no component of $S\backslash \mbox{SC}(S,\alpha)$ is a cylinder.
\end{lemma}
\begin{proof}
  Suppose there is a component $C\subset S\backslash\mbox{SC}(S,\alpha)$ which is a cylinder. Let $w_C$ and $A_C$ be the waisturve and area of $C$, respectively. The Teichm\"{u}ller maps $f_i\in SL(S,\alpha)$ in the recurrent data are affine and therefore take cylinders to cylinders. Define $C_0=C$ and $C_i = f^{-1}_i(C)$ for $i>0$. By applying the Teichm\"{u}ller deformation $g_{t_i}$ and the Teichm\"{u}ller map $f^{-1}_i$ we see that the length $w_{C_i}$ of the waistcurve of cylinder $C_i$ is $e^{-t_i}w_C$. Note that the angle $\theta_i$ between the waistcurves of $C$ and $C_i$ satisfies $\sin(\theta_i)\leq \sinh(\varepsilon_i)/\sinh(2t_i)$. By passing to the appropriate subsequences, we can control how fast the length of the waistcurves of the $C_i$ diminish as well as how small the angle is between waistcurves.

We claim that $\omega_\alpha(C_i\cap C_j) = 0$ for all $i\neq j$. Indeed, let us consider $C_1$. Since the waistcurve of $C_1$ is exponentially smaller than that of $C$ and the angle between the two foliations exponentially small (as remarked above, this can be done by passing to a subsequence if necessary), it follows that the trajectories foliating $C_1$ cannot close up if $\omega_\alpha(C\cap C_1) \neq 0$. By the same token, $\omega_\alpha(C\cap C_2) = 0$ and for the same reasons in fact $\omega_\alpha(C_1\cap C_2) = 0$. Considering this for any $i$, we have that $\omega_\alpha(C_i\cap C_j) = 0$ for all $j<i$. But if these cylinders do not overlap and their area is the same since the Teichm\"{u}ller maps preserve area, it is impossible to fit them all in $S$ since the total area is finite. It therefore follows that there is no component which is a cylinder.
\end{proof}

For a flat surface $(S,\alpha)$ with a recurrent Teichm\"{u}ller orbit, by a sequence of recurrent times, we mean a sequence of real numbers $t_i$ such that $g_{t_i}(S,\alpha)\rightarrow (S,\alpha)$.

\begin{lemma}
\label{lem:normal}
Let $(S,\alpha)$ be a flat surface of finite area whose $g_t$ orbit is recurrent and $u\in L^2_\alpha(S)$ be a real-valued function. Then for any sequence of recurrent times $\{s_i\}$ there is a sequence of affine diffeomorphisms $\{ f_i\}\subset SL(S,\alpha)$ such that the family of functions $\mathbb{F} = \{(f^{-1}_i)^*m^+_{s_i}\}_{i=0}^\infty$, where $m^+_t$ is the $\alpha_t$-meromorphic part of $u$ as in (\ref{eqn:L2tsplit}), is a normal family on $S\backslash (\Sigma\cap S)$.
\end{lemma}
\begin{proof}
To any sequence of recurrent times $s_i\rightarrow \infty$ we can associate some recurrence data. Let $\{t_i,\theta_i,\varepsilon_i,f_i\}$ be the recurrent data associated to the recurrent times $\{s_i\}$. Since $g_t(S,\alpha)$ is recurrent, it follows that for the compact set 
$$\mathcal{K}_{(S,\alpha)} = \bigcup_{\substack{\theta\in S^1 \\ t\in [0,2\varepsilon_1]}} g_t r_\theta (S,\alpha)$$
in the Teichm\"{u}ller disk of $(S,\alpha)$, we have $\Pi_{(S,\alpha)}(g_{s_i}(S,\alpha))\in \Pi_{(S,\alpha)}(\mathcal{K}_{(S,\alpha)})$. As such, for each $i$ there exists a $\varphi_i\in S^1$ such that the function $F_i \equiv (f^{-1}_i)^* m^+_{s_i}$ is meromorphic on $g_{r_i}r_{\varphi_i}(S,\alpha)$ for some $r_i\leq 2 \varepsilon_i$.

Let $K\subset S\backslash (\Sigma\cap S)$ be a compact set. Since every point in the compact set $\mathcal{K}_{(S,\alpha)}$ represents a deformation of the conformal structure of $S$, the quantity 
$$\delta_K\equiv \min d(K,\Sigma),$$
where the minimum is taken over conformal deformations of $S$ corresponding to points in $\mathcal{K}_{(S,\alpha)}$ and the distance $d$ is taken with respect to $\alpha$, is well defined. Since $\mathcal{K}_{(S,\alpha)}$ is a compact family of deformations, it follows from the Cauchy integral formula (see for example \cite[Theorem 1.2.4]{hormander}) that for any neighborhood $K'$  of $K$ in $S\backslash (\Sigma\cap S)$ there exists a constant $M_{K'}$ such that for all $F_i$ we have that
$$|F_i(z)|\leq M_{K'}\|F_i\|_{L^1_\alpha(K')} \leq M_{K'} \|F_i\|\leq M_{K'} \| u \|$$
for any $z\in K$ and therefore the functions in $\mathbb{F}$ are uniformly bounded on $K$.

Let $p\in K$ and consider a disk $D_i$ of radius $\delta_K/2$ in the conformal structure given by $g_{r_i}r_{\varphi_i}(S,\alpha)$ centered at $p$. It follows by the Cauchy integral formula that
\begin{equation}
\label{eqn:equi}
|F_i(z_1) - F_i(z_2)| \leq \frac{16M_K \|u\|}{\delta_K}|z_1-z_2|
\end{equation}
for any two points $z_1, z_2$ in a disk $D^*_i$ of radius $\delta_K/4$ in the conformal structure given by $g_{r_i}r_{\varphi_i}(S,\alpha)$ centered at $p$. Let $D_K(p)$ be a disk of radius $e^{-\varepsilon_1}\delta_K/4$ in the conformal structure of $(S,\alpha)$. Then $D_K(p)\subset D_i^*$ for all $i$. Therefore, by (\ref{eqn:equi}), $\mathbb{F}$ is equicontinuous in $D_K(p)$ and thus on $K$ since $K$ can be covered by finitely many disks of radius $e^{-\varepsilon_1}\delta_K/4$. The statement then follows from the Arzela-Ascoli theorem.
\end{proof}

Since the norm of $m^\pm_t$ is always bounded, i.e., $0\leq \|m^\pm_t\|\leq \|u\|$, then certainly 
$$\liminf \frac{d}{dt}\|m^\pm_t\|^2= \liminf \|\Im(m^\pm_t)\|^2 = 0$$
 by Lemma \ref{lem:pos}. It will be crucial that along our the sequence of recurrent times $\frac{d}{dt}\|m^\pm_{t_i}\|^2 = 4 \|\Im(m^\pm_{t_i})\|^2 \rightarrow 0$. The following lemma shows that we can always find a sequence of recurrent times  for which this is possible.

\begin{lemma}
\label{lem:im}
Let $u = \partial^\pm_t v^\pm_t + m^\mp_t\in L^2_\alpha$ be a real-valued, $X$-invariant function on a flat surface of finite area $(S,\alpha)$ whose $g_t$ orbit is recurrent. Then there is a sequence $\{t_i\}$ of recurrent times such that $\|\Im(m^\pm_{t_i})\|\longrightarrow 0$ as $t_i\longrightarrow \infty$.
\end{lemma}
\begin{proof}
By Remark \ref{rem:recurrent2} we have recurrent data  $\{\varepsilon_i,t_i,\theta_i\}$ such that $\mathrm{dist}(g_{s},g_{t_i}r_{\theta_i})\leq \varepsilon_i\rightarrow 0$ for any $s\in (t_i-e_i,t+e_i)$ for some $e_i>0$. If our sequence has the desired property, we are done. Otherwise suppose there is a subsequence $t_{i_j}$, $j\in\mathbb{N}$, and a number $\delta>0$ such that $\|\Im(m^\pm_{t_{i_j}})\|\geq\delta$ for all $j$.

Since $4\|\Im(m^\pm_t)\|^2 = \frac{d}{dt}\|m^\pm_t\|^2$ is continuous and $\|m^\pm_t\|^2$ bounded, there exists a sequence $\tau_n\rightarrow \infty$ and a further subsequence $t_{i_{j_n}}$ (to insure that $(t_{i_{j_n}}-2/\sqrt{n},t_{i_{j_n}}+2/\sqrt{n})\cap (t_{i_{j_{n+1}}}-2/\sqrt{n+1},t_{i_{j_{n+1}}}+2/\sqrt{n+1}) = \varnothing$), such that
$$|\tau_n - t_{i_{j_n}}| \leq \frac{1}{\sqrt{n}}\hspace{.5 in}\mbox{ and }\hspace{.5 in} \|\Im(m^\pm_{\tau_n})\|^2<\frac{1}{\sqrt{n}}.$$
Then
\begin{eqnarray*}
\mathrm{dist}(g_{\tau_n},g_{t_{i_{j_n}}}r_{\theta_{i_{j_n}}}) &\leq& \mathrm{dist}(g_{\tau_n},g_{t_{i_{j_n}}}) + \mathrm{dist}(g_{t_{i_{j_n}}}, g_{t_{i_{j_n}}} r_{\theta_{i_{j_n}}}) \\
&\leq& \frac{1}{\sqrt{n}} + \varepsilon_{i_{j_n}} \equiv \hat{\varepsilon}_n\longrightarrow 0.
\end{eqnarray*}
Since $g_{t_i}r_{\theta_i}\in SL(S,\alpha)$ for all $i$, $\tau_i$ is another sequence of recurrent times with the desired property.
\end{proof}
\begin{lemma}
\label{lem:norm}
Let $(S,\alpha)$ be flat surface of finite area with a recurrent $g_t$ orbit and recurrent data $\{(t_i, \theta_i, \varepsilon_i, f_i)\}_{i=1}^\infty$. Then, for any $s_i\in (t_i-e_i, t_i+e_i)$ for some $e_i>0$ as in Remark \ref{rem:recurrent2},
\begin{equation}
\label{eqn:ops}
f_i^*\partial_0^\pm = \frac{1}{\cos\theta_i}\left[\partial_{s_i}^\pm f^*_i - f^*_i(e^{s_i+t_i}Y\mp i e^{-(s_i+t_i)}X)\sin\theta_i\right] + f^*_i(\partial_0^\pm - \partial_{s_i-t_i}^\pm).
\end{equation}
Moreover it follows that if $u = \partial_t^-v_t^- + m^+_t\in L^2_\alpha$ is a real, $X$-invariant function, then
$$\partial_0^+ \left( (f^{-1}_i)^* m^+_{s_i} \right)  \longrightarrow 0$$
weakly for any sequence $s_i$ of recurrent times.
\end{lemma}
\begin{proof}
Let $\zeta\in H^1(S)$. We will now drop the indices for a while to avoid tedious notation and work under the assumption that $t$ is large while $t-s$, $\varepsilon$, and $\theta$ are small. Since we know exactly how the derivative of $f$ acts, we have
$$\partial^\pm_s f^*\zeta = f^*\left[\cos\theta \partial_{s-t}^\pm + (e^{s+t}Y \pm i e^{-(s+t)}X)\sin\theta \right]\zeta,$$
from which (\ref{eqn:ops}) follows. Using this:
\begin{eqnarray*}
((f^{-1})^*m^+_s, \partial_0^-\zeta) &=& (m^+_s,f^*\partial_0^-\zeta)\\
&=& \sec\theta (m^+_s,\partial_s^-f^*\zeta - \sin\theta f^*(e^{s+t}Y + i e^{-(s+t)}X)\zeta) \\ && \;\; + (m_s^+,f^*(\partial_0^- - \partial_{s-t}^-)\zeta)\\
&=& -\sec\theta (m^+_t,\sin\theta f^*(e^{s+t}Y + i e^{-(s+t)}X)\zeta) \\ && \;\; + (m_s^+,f^*(\partial_0^- - \partial_{s-t}^-)\zeta).
\end{eqnarray*}
Using the estimate from Remark \ref{rem:recurrent2}:
\begin{eqnarray*}
|((f^{-1})^*m^+_s, \partial_0^-\zeta)| &\leq&  e^{s+t} |\sin\theta||((f^{-1})^*m^+_t,(Y+ie^{-2(s+t)}X)\zeta)|\sec\theta \\ 
                              && \;\; + |((f^{-1})*m_s^+,(\partial_0^- - \partial_{s-t}^-)\zeta)|\\
&\leq& \frac{\sinh(2\varepsilon)}{\sinh(2t)}e^{s+t}|((f^{-1})^*m^+_t,(Y+ie^{-2(s+t)}X)\zeta)|\sec\theta\\
                              && + \|m^+_s\|\|(\partial_0^- - \partial_{s-t}^-) \zeta \| \\
&\leq& \sinh(2\varepsilon) e^{|s-t|}\|m^+_t\|\|\zeta\|_1 + 2(1-\cosh(s-t))\|m^+_s\| \|\zeta\|_1\\
&\leq& (e^{\varepsilon}\sinh(2\varepsilon) + 2(1-\cosh(\varepsilon)))\|u\|\|\zeta\|_1.
\end{eqnarray*}
Since $\varepsilon_i\rightarrow 0$, the claim follows.
\end{proof}

From Lemmas \ref{lem:pos}, \ref{lem:im}, and \ref{lem:norm}, we can get the following crucial result.

\begin{proposition}
\label{prop:cor}
Let $(S,\alpha)$ be a flat surface of finite area which is $g_t$-recurrent and $u\in L^2_\alpha(S)$ a real valued, $X$-invariant function of zero average. Then there exists a sequence of recurrent times $t_i\rightarrow\infty$ such that $f^*_i m^+_{t_i} \rightharpoonup 0$ weakly, where $m^+_{t}$ is the meromorphic part of $u$ as in (\ref{eqn:L2tsplit}), and $f_i\in SL(S,\alpha)$ are Teichm\"{u}ller maps associated to the recurrent data of $(S,\alpha)$.
\end{proposition}
\begin{proof}
Let $u$ be a real-valued, $X$-invariant function of zero average. Writing it as in (\ref{eqn:L2tsplit}),
$$u = \partial^+_tv^+_t + m_t^- = \partial ^-_t v^-_t + m^+_t.$$
Note that the norm $\| m_t^\pm\|$ is always bounded by the norm of $u$ and by Lemma \ref{lem:pos} is a non-decreasing function of $t$. Let $\{ c_i\}\equiv \left\{  (f^{-1}_i)^* m^+_{s_i} \right\}$ for $i\in\mathbb{N}$ be a sequence of functions on $(S,\alpha_0)$, where the $f_i$ and $s_i\in (t_i-e_i,t_i+e_i)$ are as in Remark \ref{rem:recurrent2}. Since $\left\|(f^{-1}_i)^* m^+_{s_i}\right\| = \|m^+_{s_i}\|\leq \|u\|$ for all $i$, the sequence $\{c_i\}$ is bounded. Therefore, there exists a function $m_*^+$ and weakly convergent subsequence $\{ c_{i_j}\}$ such that $(f^{-1}_{i_j})^* m^+_{s_{i_j}} \rightharpoonup m^+_*$.

We can assume, by Lemma \ref{lem:im}, that (since $f^*_{i}$ is unitary)
\begin{equation}
\label{eqn:im2}
\|\Im\left( (f^{-1}_i)^* m^+_{s_i} \right) \|\longrightarrow 0
\end{equation}
as $i\rightarrow\infty$. By Lemma \ref{lem:norm}, $m^+_*$ is meromorphic
 and by (\ref{eqn:im2}) it has zero imaginary part and thus it is a constant. Since $u$ has zero average, $\int_S m_t\,\omega_\alpha = 0$ for all $t$. It follows from this that 
$$\int_S (f^{-1}_i)^*m^+_{s_i}\,\omega_\alpha = 0$$
for all $i$ since the Jacobian of $f_i$ is identically 1 for all $i$. Thus, since $m^+_*$ is a constant of zero average, it is identically zero.
\end{proof}

\begin{proposition}
\label{prop:rec}
Let $(S,\alpha)$ be a flat surface of finite area whose $g_t$ orbit is recurrent. Then the translation flow is ergodic.
\end{proposition}
\begin{proof}
Let $u$ be a real-valued, $X$-invariant function of zero average. Writing it as in (\ref{eqn:L2tsplit}),
$$u = \partial^+_tv^+_t + m_t^- = \partial ^-_t v^-_t + m^+_t.$$

Let $t_i$ be a sequence of return times such that the conclusion of Proposition \ref{prop:cor} holds. Consider an exhaustion $K_0\subset K_1\subset K_2\subset \cdots $ of $S\backslash (\Sigma\cap S) = \bigcup K_n$ by compact sets $K_n$ such that $\omega_\alpha (K_n) \geq 1-\frac{1}{n}$ and consider the sequences of functions $F^n_k = (f_{n_k}^{-1})^* m_{t_{n_k}}$ which, by Lemma \ref{lem:normal}, converge uniformly on $K_n$. By Proposition \ref{prop:cor}, for each $n$, $F^n_k\rightarrow 0$ uniformly as $k\rightarrow \infty$ on $K_n$. Therefore, for every $n$ and $\delta > 0$, there exists an $N_\delta$ such that
\begin{equation}
\label{eqn:lbd}
\|(f^{-1}_{n_k})^* m^+_{t_{n_k}} \|_{L^\infty_\alpha(K_n)} < \delta 
\end{equation}
for all $k>N_\delta$. Equivalently,
$$B_{n,k}\equiv \| m^+_{t_{n_k}} \|_{L^\infty_\alpha(f_{n_k}(K_n))}< \delta.$$
Now let $\varepsilon >0$ and choose $n$ big enough so that $\omega_\alpha(S\backslash K_n)< \varepsilon$. Then
\begin{equation}
\label{eqn:superbound}
\begin{split}
\|m^+_{t_{n_k}}\|^2 &= \int_{K_n}(f^{-1}_{n_k})^*(m^+_{t_{n_k}} u)\,\omega_\alpha + \int_{S\backslash K_n} (f^{-1}_{n_k})^*(m^+_{t_{n_k}} u)\,\omega_\alpha \\
&= \int_{f_{n_k}(K_n)}m^+_{t_{n_k}} u\,\omega_\alpha + \int_{f_{n_k}(S\backslash K_n)} m^+_{t_{n_k}} u\,\omega_\alpha \\
&\leq B_{n,k}\|u\|_{L^1_\alpha(K_n)} + \|u\|_\infty \int_{f_{n_k}(S\backslash K_n)} |m^+_{t_{n_k}}| \,\omega_\alpha \\
&\leq \varepsilon \|u\|_{L^1_\alpha(K_n)} +  \omega_\alpha(S\backslash K_n)^{\frac{1}{2}} \|m^+_{t_{n_k}}\| \|u\|_\infty \\
&\leq \sqrt{\varepsilon}(\sqrt{\varepsilon}+\|u\|_\infty) \|u\|
\end{split}
\end{equation}
for $k>N_\varepsilon$ as in (\ref{eqn:lbd}). This implies that $\|m^+_t\|$ can be arbitrarily small for arbitrarily large values of $t$. It follows by Lemma \ref{lem:pos} that $m^+_t \equiv 0$ for all $t$. Moreover we have $u = \partial_t^\pm v_t^\pm$ for some $v^\pm_t \in H^1_\alpha(S)$. Since $u$ is real and, by Lemma \ref{lem:ids}, $v^\pm_t$ imaginary, $u = \partial^\pm_t v^\pm_t = X_tv^\pm_t \pm i Y_t v^\pm_t$ implies that $v^\pm_t$ is $X$-invariant.

For a point $p\in S\backslash\Sigma$ and $w,h >0$, a \emph{$(w,h)$-rectangle} for $p$ is defined as
$$K_p(w,h) = \bigcup_{\substack{s\in(-w,w)\\t\in(-h,h)}}  \varphi^X_s\circ\varphi^Y_t(p),$$
where $\varphi^{X,Y}$ are the respective flows generated by $X$ and $Y$. It is well defined for any $p\in S\backslash\Sigma$ if $w$ and $h$ are chosen small enough. Since $v^\pm_t$ is $X$-invariant, its restriction to any $(w,h)$-rectangle $K_p(w,h)$ for some point $p$ is a function of one variable, namely, the $Y$-coordinate. Thus by the Sobolev embedding theorem we have that since $v^\pm_t\in H^1_\alpha(K_p(w,h))$, $v^\pm_t$ is a continuous function on $\overline{K_p(w,h)}$. 

Let $\{K_{p_i}(w_i,h_i) \}_{i\in \mathbb{N}}$ be an open cover of $S\backslash\Sigma$. By the Sobolev embedding theorem $v^\pm_t$ is continuous on each $\overline{K_{p_i}(w_i,h_i)}$, and so it is continuous on $S$. If $\varphi^X_t$ is minimal, since $v^\pm_t$ is $X$-invariant and continuous, then $v^\pm_t$ is constant and thus $u=0$ and we conclude that the flow is ergodic.

Otherwise supposed that the flow is not minimal. Since by Lemma \ref{lem:cylinder} there are no cylinders, then by \cite[Theorem 13.1]{strebel:book} and Remark \ref{rem:measure} the set of orbits of infinite length which are recurrent in forwards and backwards time form a set of full measure. Let $p$ be a recurrent point and $\beta_p$ an orthogonal transversal with $p$ as and endpoint. Since $p$ is a recurrent point, there exists a sequence of times $t_k\rightarrow \infty$ such that 
\begin{equation}
\label{eqn:transversal}
 \lim_{k\rightarrow\infty} \varphi^X_{t_k}(p) = p \hspace{.5 in} \mbox{ and } \hspace{.5 in} \varphi^X_{t_k}(p) \cap \beta_p \neq \varnothing 
\end{equation}
for all $k$, and the distance between $p$ and $\varphi^X_{t_k}(p)$ is decreasing. Therefore we can compute the value of $u$ at $p$. Since $v_t^\pm$ is continuous and $\varphi^X$-invariant, for any fixed $\tau>0$
\begin{equation}
\label{eqn:recurrentPoint}
\begin{split}
u(p) = \pm i Y_\tau v_\tau^\pm(p) &= \lim_{k\rightarrow\infty} -\frac{v_\tau^\pm(p) - v_\tau^\pm\circ \varphi^X_{t_k}(p)}{|p - \varphi^X_{t_k}(p)|} \\
             &= \lim_{k\rightarrow\infty} -\frac{v_\tau^\pm(p) - v_\tau^\pm(p)}{|p -\varphi^X_{t_k}(p)|} = 0.
\end{split}
\end{equation}
Since $p$ was recurrent and the recurrent points form a set of full measure, $u$ is constant almost everywhere.
\end{proof}

\begin{proof}[Proof of Theorem \ref{thm:main}]
If the $g_t$ orbit of the flat surface $(S,\alpha)$ is recurrent, then the theorem is proved by Proposition \ref{prop:rec}. Therefore it remains to prove the theorem for flat surfaces which are not recurrent but nonetheless have a limit point $\ell$ in a compact set $\Lambda\subset H_{(S,\alpha)}$.

Consider a fundamental domain $\hat{\Lambda}$ of the action of $SL(S,\alpha)$ on $SL(2,\mathbb{R})/SO(2,\mathbb{R})$, let $\hat{\ell}$ be the point on this domain which projects to $\ell$: $\ell = \Pi_{(S,\alpha)} (\hat{\ell})$ and consider $S\in\hat{\Lambda}$ representing $(S,\alpha)$. There exist numbers $s,t\in\mathbb{R}$ such that $h^tg_s S = \hat{\ell}$. Consider the flat surface $h^tg_s(S,\alpha)$. It has the following properties:
\begin{enumerate}
\item It is in the stable horocycle of the $g_T$ orbit of $(S,\alpha)$. Therefore, the distance on $SL(2,\mathbb{R})/SL(S,\alpha)$ between $g_{T+s}(S,\alpha)$ and $g_T h^t g_s(S,\alpha)$ goes to zero as $T\rightarrow \infty$ since $SL(S,\alpha)$ acts by isometries.
\item The horizontal foliation of $(S,\alpha)$ and that of $h^t g_s (S,\alpha)$ are the same. This follows from the fact that $g_s$ and $h^t$ parametrize the stable horocycle of any point in $SL(2,\mathbb{R})$, meaning that the horizontal foliation of any point in the stable horocycle limits to the same projective horizontal foliation under the geodesic (Teichm\"{u}ller) flow.
\end{enumerate}
It follows from property (i) above that the $g_T$ orbit of $h^t g_s (S,\alpha)$ is recurrent. Indeed, since there is a sequence of times $\{T_i\}_0^\infty$ such that $g_{T_i}\rightarrow \ell$ and $g_T h^t g_s$ is asymptotic to $g_T$ as $T\rightarrow \infty$, it follows that $h^t g_s (S,\alpha)$ has a recurrent $g_T$ orbit. Therefore, by Proposition \ref{prop:rec}, the horizontal foliation of $h^t g_s (S,\alpha)$ is ergodic. Moreover, by property (ii) above, the horizontal foliation of $h^t g_s (S,\alpha)$ is the same as that of $(S,\alpha)$. Therefore, since it is ergodic for $h^t g_s (S,\alpha)$, it is ergodic for $(S,\alpha)$.
\end{proof}

\begin{definition}
The $g_t$ orbit of $(S,\alpha)$ is \emph{periodic} if there exists an $s$ such that $g_s\in SL(S,\alpha)$. The number $s$ is the \emph{period} of $(S,\alpha)$.
\end{definition}

Suppose $(S,\alpha)$ is $g_t$ periodic with period $T$. Then there exists a unique affine diffeomorphism $f:S\rightarrow S$ such that $Df$ can be identified with $r\equiv g_T\in SL(S,\alpha)$.
Any periodic orbit $g_t(S,\alpha)$ is obviously recurrent and thus by Theorem \ref{thm:main} has an ergodic horizontal foliation. By considering the orbit $g_{-t}r_{\pi/2}(S,\alpha)$ for $t\geq 0$ and (\ref{eqn:ops}), then the same theorem gives us the ergodicity of the vertical foliation.
\begin{corollary}
\label{cor:periodic}
Let $(S,\alpha)$ be a flat surface of infinite genus and finite area which is $g_t$-periodic. Then the flows generated by $X$ and $Y$ on $S$ are ergodic.
\end{corollary}
If $SL(S,\alpha)$ is a lattice (meaning $SL(2,\mathbb{R})/SL(S,\alpha)$ has finite volume), then by the Poincar\'{e} recurrence theorem, almost every orbit is recurrent. This translates into the following corollary.
\begin{corollary}
Let $(S,\alpha)$ be a flat surface of finite area such that $SL(S,\alpha)$ is a lattice. Then the translation flow is ergodic in almost every direction.
\end{corollary}

Veech \cite{veech:dich} was the first to notice that if the group of affine automorphisms of a surface (now known as the Veech group) is big enough, then the translation flow on it is reminiscent to the case on the flat torus: it is either completely periodic or uniquely ergodic. This dichotomy is referred to as \emph{the Veech dichotomy}. More specifically, for a Veech surface, that is, for a closed flat surface of finite genus for which $SL(S,\alpha)$ is a lattice in $SL(2,\mathbb{R})$, this dichotomy holds.

A modern proof of the Veech dichotomy hardly relies on the fact that it is coming from a surface of finite genus. It does depend, however, on the nice characterization of limit sets of trajectories available when the flat surface is compact and of finite genus. Suppose that the Veech group of a flat surface of infinite genus and finite area $(S,\alpha)$ is a lattice. If the $g_t$ orbit does not leave every compact set of $SL(2,\mathbb{R})/SL(S,\alpha)$, the horizontal foliation of $(S,\alpha)$ is ergodic by Theorem \ref{thm:main}. 

If the $g_t$ orbit leaves every compact subset of $H_{(S,\alpha)}$, then $g_t(S,\alpha)$ limits to a cusp of $H_{(S,\alpha)}$ and, therefore (see \cite[\S 1.3-1.4]{pascal:veech}), the horizontal foliation of $(S,\alpha)$ is preserved by a parabolic element of $h\in SL(S,\alpha)$. Having deeper knowledge of the structure of limit sets of trajectories may help in completing the proof for the Veech dichotomy in the case of flat surfaces of infinite genus and finite area by following the proof for compact surfaces (see \cite{pascal:veech} for a proof).

\begin{question}
\label{2}
Is there a flat surface of infinite genus and finite area whose Veech group is of the first kind? Is there one such surface whose Veech group is a lattice? If there is, is it true that the horizontal foliation is ergodic or completely periodic?
\end{question}

\section{An Ergodicity Criterion}
\label{sec:diverging}

In this section we will give the proof of Theorem \ref{thm:integrability2}. We will use the same notation as in the previous section, namely the notation $g_t$ used to express Teichm\"{u}ller deformation. 

\begin{proof}[Proof of Theorem \ref{thm:integrability1}]
Let $u\in L^2_\alpha(S)$ be a real-valued, $X$-invariant function of zero average. By (\ref{eqn:L2tsplit}), we can write as
$$u = \partial^\pm_t v^\pm_t + m^\mp_t,$$
and $m^\pm_t = R_t\pm iI_t$. We will start by showing, as in Proposition \ref{prop:rec}, that $u$ is constant when we assume (\ref{eqn:integrability}). Fix $\eta>0$.

For a fixed $t>0$, consider a point $z^*\in S_{\varepsilon(t), t}$. By the Poisson integral formula we have
$$X_tI_t(z_t) = \frac{1}{2\pi}\int_0^{2\pi}\left( \frac{Re^{i\tau}}{(Re^{i\tau} - z_t)^2} + \frac{Re^{-i\tau}}{(Re^{-i\tau} - \bar{z}_t)^2} \right) I_t(Re^{i\tau})\, d\tau,$$
where $z_t$ is a local coordinate inside a $\alpha_t$-disk of radius $R$ centered at $z^*$ for some $R<\epsilon(t)$. In particular, for $z_t = z^*$,
 $$X_tI_t(z^*) = \frac{1}{\pi}\int_0^{2\pi}R^{-1}\cos(\tau)I_t(Re^{i\tau})\, d\tau.$$
Integrating both sides of this expression,
\begin{eqnarray*}
\int_{\varepsilon(t)/3}^{2\varepsilon(t)/3}X_tI_t(z^*)R^2 \, dR &=& \frac{1}{\pi}  \int_{\varepsilon(t)/3}^{2\varepsilon(t)/3} \int_0^{2\pi}\cos(\tau)I_t(Re^{i\tau})\, R d\tau dR  \\
&=& \frac{7\varepsilon(t)^3}{81}X_tI_t(z^*),
\end{eqnarray*}
from which we can compute the bound
\begin{equation}
\label{eqn:derBnd}
\begin{split}
|X_t I_t(z^*)| &\leq \frac{81}{7\varepsilon(t)^3\pi}\|I_t\|_{L^2\left(A_{\frac{\varepsilon(t)}{3},\frac{2\varepsilon(t)}{3}}\right)}\omega_{\alpha}\left(A_{\frac{\varepsilon(t)}{3},\frac{2\varepsilon(t)}{3}}\right)^{\frac{1}{2}}  \\ 
&\leq \frac{4}{\varepsilon(t)^2} \|I_t\|_{L^2\left(A_{\frac{\varepsilon(t)}{3},\frac{2\varepsilon(t)}{3}}\right)}  \leq \frac{4}{\varepsilon(t)^2} \|I_t\|,
\end{split}
\end{equation}
where $A_{a,b}$ is the annulus of inner and outer radius $a$ and $b$, respectively. The same computation yields the same bound for $Y_t I_t (z^*)$. Since these bounds only depended on the fact that $z^*\in S_{\varepsilon(t),t}$, by the Cauchy-Riemann equations,
\begin{equation}
\label{eqn:gradient}
\begin{split}
\|\nabla_t I_t\|_{L^\infty(S_{\varepsilon(t),t})} &\leq \frac{8}{\varepsilon(t)^2} \|I_t\|_{L^2_\alpha(S)}, \\
\|\nabla_t R_t\|_{L^\infty(S_{\varepsilon(t),t})} &\leq \frac{8}{\varepsilon(t)^2} \|I_t\|_{L^2_\alpha(S)},
\end{split}
\end{equation}
where $\nabla_t$ denotes the gradient with respect to $\alpha_t$.

Any point $z\in (S\backslash S_{\varepsilon(t),t})\backslash\Sigma$ at a distance $\varrho< \varepsilon(t)$ from $\Sigma$, by (\ref{eqn:derBnd}), satisfies
$$| \nabla R_t(z)|\leq \frac{8}{\varrho^2}\|I_t\| \hspace{.3 in }\mbox{ and }\hspace{.3 in } | \nabla I_t(z)|\leq \frac{8}{\varrho^2}\|I_t\|.$$

Let $z_i\in S^i_t, z_j\in S_t^j$ with $i\neq j$ and $\gamma:[0,1] \rightarrow S$ be a smooth path from $z_i$ to $z_j$ with $\gamma^{-1}(\gamma([0,1])\cap S\backslash S_{\varepsilon(t),t}) = (a,b)$, for some $0<a<b<1$, $\gamma([0,a])\subset S_i$, $\gamma([b,1])\subset S_j$ and $\mbox{dist}_t(\gamma([0,1]),\Sigma)\geq \delta_t$. By the estimate above,
\begin{equation}
\label{eqn:thinBd}
|R_t(z_i) - R_t(z_j)|   \leq 8 \int_0^1 \left| \frac{ds}{\mbox{dist}_t(\gamma(s),\Sigma)^2}\right| \|I_t\| \leq \frac{16}{\delta_t}\|I_t\|.
\end{equation}
The same bound holds for $|I_t(z_1) - I_t(z_2)|$.

Let $z_1, z_2\in S_{\varepsilon(t), t}$. Combining (\ref{eqn:gradient}) and (\ref{eqn:thinBd}),
\begin{equation}
\label{eqn:Rbound}
|R_t(z_1) - R_t(z_2)| \leq 16 \left(\frac{1}{\varepsilon(t)^2}\sum_{i=1}^{C_t} \mathcal{D}_t^i + \frac{C_t-1}{\delta_t} \right)\|I_t \|.
\end{equation}
The same bound holds for $|I_t(z_1)-I_t(z_2)|$.

We claim that 
\begin{equation}
\label{eqn:vanish}
\liminf_{t\rightarrow \infty} \left(\frac{1}{\varepsilon(t)^2}\sum_{i=1}^{C_t} \mathcal{D}_t^i + \frac{C_t-1}{\delta_t} \right)\|I_t \| = 0.
\end{equation}
Otherwise, there exists a $\lambda >0 $ such that 
$$0<\lambda < \left(\frac{1}{\varepsilon(t)^2}\sum_{i=1}^{C_t} \mathcal{D}_t^i + \frac{C_t-1}{\delta_t} \right)\|I_t \|$$
for all $t$ or, equivalently, 
$$\lambda \left(\frac{1}{\varepsilon(t)^2}\sum_{i=1}^{C_t} \mathcal{D}_t^i + \frac{C_t-1}{\delta_t} \right)^{-1} < \|I_t\|.$$
Squaring both sides and integrating with respect to $t$, by (\ref{eqn:integrability2}) and Lemma \ref{lem:pos},
$$\infty = \lambda^2\int_0^\infty  \left(\frac{1}{\varepsilon(t)^2}\sum_{i=1}^{C_t} \mathcal{D}_t^i + \frac{C_t-1}{\delta_t} \right)^{-2} \, dt \leq \int_0^\infty \|I_t\|^2\, dt \leq \|u\|^2<\infty,$$
a contradiction. Therefore (\ref{eqn:vanish}) holds.

It follows from (\ref{eqn:Rbound}) and (\ref{eqn:vanish}) that there exist arbitrary large values of $t$ such that 
\begin{equation}
\label{eqn:RinfBound}
\|m^\pm_t\|_{L^\infty(S_{\varepsilon(t),t})} < \eta.
\end{equation}

We can now bound $\|m_t^\pm\|$ as in (\ref{eqn:superbound}). Let $t_n\rightarrow \infty$ be a subsequence such that 
$$ \left(\frac{1}{\varepsilon(t_n)^2}\sum_{i=1}^{C_{t_n}} \mathcal{D}_{t_n}^i + \frac{C_{t_n}-1}{\delta_{t_n}} \right)\|I_{t_n} \| \rightarrow 0$$
and choose $n$ big enough so that, by (\ref{eqn:RinfBound}), $\|m^\pm_{t_n}\|_{L^\infty(S_{\varepsilon(t_n),t_n})} < \eta$. As in (\ref{eqn:superbound}),
\begin{equation}
\label{eqn:superbound2}
\begin{split}
\|m^+_{t_n}\|^2 &= \int_{S_{\varepsilon(t_n),t_n}}m^+_{t_n} u\,\omega_\alpha + \int_{S\backslash S_{\varepsilon(t_n),t_n}} m^+_{t_n} u\,\omega_\alpha \\
     &\leq \|m^\pm_{t_n}\|_{L^\infty(S_{\varepsilon(t_n),t_n})} \|u\|_{L^1_\alpha(S_{\varepsilon(t_n),t_n})} + \|u\|_\infty \int_{S\backslash S_{\varepsilon(t_n),t_n}} |m^+_{t_n}| \,\omega_\alpha \\
&\leq \eta \|u\|_{L^1_\alpha(S_{\varepsilon(t_n),t_n})}+  \omega_\alpha(S\backslash S_{\varepsilon(t_n),t_n})^{\frac{1}{2}} \|m^+_{t_n}\| \|u\|_\infty \\
&\leq \sqrt{\eta}(\sqrt{\eta}+\|u\|_\infty) \|u\|.
\end{split}
\end{equation}
Therefore we can bound $\|m_t^\pm\|$ by any $\eta>0$ for arbitrarily large values of $t$. It follows by Lemma \ref{lem:pos} that $m^+_t \equiv 0$ for all $t$. Moreover we have $u = \partial_t^\pm v_t^\pm$ for some $v^\pm_t \in H^1_\alpha(S)$. Since $u$ is real and, by Lemma \ref{lem:ids}, $v^\pm_t$ imaginary, $u = \partial^\pm_t v^\pm_t = X_tv^\pm_t \pm i Y_t v^\pm_t$ implies that $v^\pm_t$ is $X$-invariant.

For a point $p\in S\backslash\Sigma$ and $w,h >0$, a \emph{$(w,h)$-rectangle} for $p$ is defined as
$$K_p(w,h) = \bigcup_{\substack{s\in(-w,w)\\t\in(-h,h)}}  \varphi^X_s\circ\varphi^Y_t(p),$$
where $\varphi^{X,Y}$ are the respective flows generated by $X$ and $Y$. It is well defined for any $p\in S\backslash\Sigma$ if $w$ and $h$ are chosen small enough. Since $v^\pm_t$ is $X$-invariant, its restriction to any $(w,h)$-rectangle $K_p(w,h)$ for some point $p$ is a function of one variable, namely, the $Y$-coordinate. Thus by the Sobolev embedding theorem we have that since $v^\pm_t\in H^1_\alpha(K_p(w,h))$, $v^\pm_t$ is a continuous function on $\overline{K_p(w,h)}$. Let $\{K_{p_i}(w_i,h_i) \}$ be an open cover of $S\backslash\Sigma$. By the Sobolev embedding theorem $v^\pm_t$ is continuous on each $\overline{K_{p_i}(w_i,h_i)}$, and so it is continuous on $S$. 

The starting integrability assumption (\ref{eqn:integrability}) forbids the horizontal foliation of $(S,\alpha)$ to have a cylinder. Otherwise, if $(S,\alpha)$ had a cylinder, either $\varepsilon(t)$ or $\delta_t$ (or both) need to exhibit exponential decay in order to satisfy condition (i), but this comes at the expense of making the integral in (\ref{eqn:integrability}) diverge.

Since there are no cylinders and, by assumption, the set of points whose orbits diverge have negligible measure, by the Poincar\'{e} recurrence theorem, almost every orbit is recurrent. For a recurrent point $p$, let $\beta_p$ be an orthogonal transversal with $p$ as an endpoint satisfying (\ref{eqn:transversal}). We can calculate the value of $u(p)$ by the computation in (\ref{eqn:recurrentPoint}). Therefore, for $\omega_\alpha$-almost every point, $u(p) = 0$ and the horizontal flow is ergodic with respect to the Lebesgue measure $\omega_\alpha$.
\end{proof}

\begin{proof}[Proof of Theorem \ref{thm:integrability1}]
The proof that (\ref{eqn:integrability}) implies ergodicity in this case follows the proof of Theorem \ref{thm:integrability2} by letting $\varepsilon(t) \equiv \varepsilon_0$ for some fixed and small $\varepsilon_0$, letting the set $S_{\varepsilon(t),t}$ be defined as
$$S_{\varepsilon(t),t}\equiv \{z\in S: \mbox{dist}_t(z,\Sigma)\geq \varepsilon_0 \},$$
where $\mbox{dist}_t$ is the distance on $S$ with respect to the metric given by $\alpha_t$, by noting that the number of components of $S_{\varepsilon(t), t}$ is uniformly bounded in $t$, and by noting that by a result of Masur and Smillie \cite[Corollary 5.6]{MasurSmillie-dim}, there is a constant $K$, depending only of the stratum of Abelian differentials to which $\alpha$ belongs, such that if $D_t$ is the diameter of $(S,\alpha_t)$, $\delta_t(S,\alpha)$ the systole of $(S,\alpha_t)$, and $D_t(S,\alpha)>\sqrt{2/\pi}$, then 
\begin{equation}
D_t(S,\alpha)\leq \frac{K}{\delta_t(S,\alpha)}.
\end{equation}
As such, the quantity $\delta_t$ in Theorem \ref{thm:integrability2} becomes half the systole $\delta_t$ in Theorem \ref{thm:integrability1}. It then follows that condition (\ref{eqn:integrability2}) becomes (\ref{eqn:integrability}) in this special case. We leave the details to the reader.

It remains to show that the Lebesgue measure $\omega_\alpha$ is the only invariant measure for the translation flow. Let us suppose that $\mu$ is another invariant measure. Since the flow is minimal, $\mu$ is not atomic. The measures $\omega_\alpha$ and $\mu$ are mutually singular. Define the one parameter family of $X$-invariant measures $\mu_s = s\omega_\alpha + (1-s)\mu$ for $s\in [0,1]$. If we show that $\mu_s$ is an ergodic measure for some $s\in(0,1)$, then, since $\mu_s$ is a convex combination of ergodic measures, this will contradict the fact that ergodic measures are extremal points in the Choquet simplex of invariant measures and therefore imply unique ergodicity.

Let $s\in (0,1)$ be fixed. For $p\in S\backslash\Sigma$, define $L_{\eta}^\alpha(p)$ to be the leaf of the vertical foliation of $\alpha$ of length $\eta$ (with respect to the metric compatible with $\alpha$) with $p$ at its ``bottom'' point. In other words,
$$L^\alpha_{\eta}(p) = \bigcup_{\tau\in[0,\eta]}\phi^Y_\tau(p).$$
Let $R_{\varepsilon,\eta}^\alpha(p)$ be the rectangle with $p$ at its lower-left corner with vertical and horizontal side-lengths $\eta$ and $\varepsilon$, respectively, with respect to the metric compatible with $\alpha$. More precisely,
$$R_{\varepsilon,\eta}^\alpha(p) = \bigcup_{\substack{r\in[0,\varepsilon]\\ \tau\in[0,\eta]}} \phi^X_r\circ\phi^Y_\tau(p).$$
Without loss of generality we implicitly assume that $\varepsilon$ and $\eta$ are small enough such that $R_{\varepsilon,\eta}(p)\cap\Sigma = \varnothing$.

We define a transverse measure to $\mathcal{F}_\alpha^h$, $\Upsilon_s$, defined by
$$\Upsilon_s(L_{\eta}^\alpha(p)) = \lim_{\varepsilon\rightarrow0}\frac{1}{\varepsilon}\mu_s(R_{\varepsilon,\eta}^\alpha(p)).$$
It follows that
$$\mu_s(R_{\varepsilon,\eta}^\alpha(p)) = s\varepsilon\eta + (1-s)\mu(R_{\varepsilon,\eta}(p))\geq s\varepsilon\eta$$
and therefore that
\begin{equation}
\label{eqn:vertLength}
\Upsilon_s(L_\eta^\alpha(p)) = s\eta + \lim_{\varepsilon\rightarrow 0}\frac{1}{\varepsilon}(1-s)\mu(R_{\varepsilon,\eta}(p))\geq s\eta.
\end{equation}

Construct the homeomorphism $\Phi_s: S\rightarrow S$ as follows. For a point in $z\in S\backslash \Sigma$ and using $w = x+iy$ as a local coordinate (identifying $z$ with zero in these local coordinates), the map $\Phi_s$ takes the local form 
$$\Phi_s:w\mapsto \Phi_s(w) = x+sign(y)\cdot i \Upsilon_s(L_{y}^\alpha(z)).$$ 
The map $\Phi_s$ induces a map on the measure $\mu_s$ which sends it to $\mu^s_*\equiv (\Phi_s)_*\mu_s $, the Lebesgue measure in $\Phi_s(S)$. Moreover, the map induces a new translation structure on $S$ and preserves the smooth foliation $\mathcal{F}_\alpha^h$. Considering the transverse foliations $\mathcal{F}^h_\alpha$ and $\mathcal{F}^v_\alpha$ with respect to this new translation structure obtained through the map $\Phi_s$, by \cite{HubbardMasur}, there exists a unique Abelian differential $\alpha_s$ such that $\mathcal{F}^{h,v}_\alpha = \mathcal{F}^{h,v}_{\alpha_s}$. By construction, $\mu_*^s = \omega_{\alpha_s}$.


We now compare the systole $\delta_t$ of $(S,\alpha_t)$ with the systole $\delta_t^s$, the systole of the flat surface $g_t(S,\alpha_s)$. For some fixed $t>0$, let $\ell_t \in SC(S,e^{i\theta}\alpha_t)$, for some $\theta$, such that $\mathrm{length}_{\alpha_t}(\ell_t) = \delta_t$, and let $\ell^s_t = \Phi_s(\ell_t) \in SC(S,e^{i\theta} g_t \alpha_s)$. Note that this implies that $\mathrm{length}_{g_t \alpha_s}(\ell_t^s) = \delta_t^s$.

Since $\Phi_s$ preserves the smooth foliation $\mathcal{F}^h_\alpha$ along with its smooth transverse measure, $V_{\alpha_t}(\ell_t) = V_{g_t \alpha_s}(\ell^s_t)$. On the other hand, by (\ref{eqn:vertLength}), $H_{g_t \alpha_s}(\ell^s_t)\geq sH_{\alpha_t}(\ell_t)$. Therefore, by (\ref{eqn:length}), we have the bound $\delta_t^s \geq s \delta_t$. Since this argument works for any $t$, the bound holds for all $t$. By (\ref{eqn:integrability}),
$$\int_0^\infty (\delta_t^s)^2\, dt \geq s^2 \int_0^\infty \delta_t^2\, dt = \infty,$$
and therefore the horizontal flow on the flat surface $(S,\alpha_s)$ is ergodic with respect to the Lebesgue measure $\omega_{\alpha_s} = \mu_*^s = s(\Phi_s)_*\omega_\alpha + (1-s)(\Phi_s)_*\mu$. But this contradicts that ergodic measures are not convex combinations of other ergodic measures. Therefore, $\omega_\alpha$ is the only $X$-invariant probability measure.
\end{proof}

\bibliographystyle{amsalpha}
\bibliography{biblio}
\end{document}